\documentclass[12pt]{article}
\usepackage{latexsym,amssymb,amsmath,amsfonts,amsthm}
\newtheorem{theorem}{Theorem}

\newtheorem{lemma}{Lemma}[section]

\newtheorem{conj}{Conjecture}

\def\beq{ \begin{equation} }
\def\eeq{ \end{equation} }
\def\mn{\medskip\noindent}

\def\ep{\epsilon}

\def\square{\vcenter{\vbox{\hrule height .4pt
  \hbox{\vrule width .4pt height 5pt \kern 5pt
        \vrule width .4pt} \hrule height .4pt}}}

\def\TT{\mathbb{T}}
\def\ZZ{\mathbb{Z}}

\def\clearp{}

\usepackage{graphicx}
\graphicspath{%
    {converted_graphics/}
    {/}
}
\begin{document}

\title{The Contact Process on Random Graphs 
\\and Galton-Watson Trees}
\author{Xiangying Huang and Rick Durrett \\
Dept.~of Math, Duke University}

\date{\today}						

\maketitle

\begin{abstract}
The key to our investigation is an improved (and in a sense sharp) understanding of the survival time of the contact process on star graphs. Using these results, we show that for the contact process on Galton-Watson trees, when the offspring distribution (i) is subexponential the critical value for local survival $\lambda_2=0$ and (ii) when it is geometric($p$) we have $\lambda_2 \le C_p$, where the $C_p$ are much smaller than previous estimates. We also study the critical value $\lambda_c(n)$ for ``prolonged persistence'' on graphs with $n$ vertices generated by the configuration model. In the case of power law and stretched exponential distributions where it is known $\lambda_c(n) \to 0$ we give estimates on the rate of convergence. Physicists tell us that $\lambda_c(n) \sim 1/\Lambda(n)$ where $\Lambda(n)$ is the maximum eigenvalue of the adjacency matrix. Our results show that this is accurate for graphs with power-law degree distributions, but not for stretched exponentials.
\end{abstract}

\section{Introduction}

In the contact process on a graph $G$, occupied sites become vacant at rate 1, and give birth onto vacant neighbors at rate $\lambda$.
Harris \cite{Harris} introduced the contact process on $G=\ZZ^d$ in 1974. The state at time $t$ is $\xi_t \subset \ZZ^d$.
It is often thought of as a model for the spread of species. In this case $\xi_t$ is the set of occupied sites, and sites in $\xi_t^c$ are vacant.
However, it can also be viewed as a spatial SIS epidemic model. In this case $\xi_t$ is the set of infected sites, and sites in $\xi_t^c$ are susceptible.
Both interpretations are common in the literature, so the reader will see both here.

Let $\xi_t^0$ be the process starting from only the origin occupied and let $\xi^1_t$ be the process starting from all sites occupied.
Harris introduced the critical value 
$$
\lambda_c = \inf\{ \lambda : P( \xi^0_t \neq \emptyset \hbox{ for all $t$} )  >0 \},
$$
and proved that on $\ZZ^d$ we have $0 < \lambda_c < \infty$. He also showed that for $\lambda > \lambda_c$, $\xi^1_t$ converges to a limit that is a nontrivial stationary distribution. A rich theory has been developed for the contact process on $\ZZ^d$. See Liggett's 1999 book \cite{Lig99} for a summary of much of what is known.

Pemantle \cite{Pem92} was the first to study the contact process on the tree $\TT^d$ in which each vertex has degree $d+1$.
Here, and in what follows, we assume $d\ge 2$ since $\TT_1=\ZZ$. Let $0$ be the root of the tree and let $P_0$ be the probability measure for the process starting from only the root occupied. Pemantle found that the contact process on $\TT^d$ has two critical values. 
\begin{align*}
\lambda_1&  = \inf\{ \lambda : P_0( \xi_t \neq \emptyset \hbox{ for all $t$}) > 0 \}, \\
\lambda_2&  = \inf\{ \lambda :  \liminf_{t\to\infty} P_0( 0 \in \xi_t) > 0 \}.
\end{align*}

\noindent
By deriving bounds on the critical values, he showed that $\lambda_1 < \lambda_2$ when $d \ge 3$. Liggett \cite{Lig96} 
settled the case $d=2$ by showing $\lambda_1 < 0.605 < 0.6609 < \lambda_2$. At about the same time, Stacey \cite{Sta96} gave a proof that $\lambda_1 < \lambda_2$ that did not rely on bounds. The stationary distributions and limiting behavior of the contact process on trees is an interesting subject that has been extensively studied. See Liggett's book \cite{Lig99} for an account of the results. 

\subsection{Results for star graphs}

Let $G_k$ be a star graph with center 0 and leaves $1, 2, \ldots, k$ and let $\xi_t$ be the set
 of vertices infected in the contact process at time $t$.
Write the state $\xi_t$ as $(i,j)$ where $i$ is the number of infected leaves and
$j=1$ if the center is infected and $j=0$ otherwise. We write $P_{i,j}$ for the law of the process
starting from $(i,j)$. 
Pemantle \cite{Pem92} was the first to study the persistence time of the contact process on stars.
See his Section 4. He did his analysis on the ``ladder graph'' $\{0, \ldots , n \} \times \{0,1\}$ so he ended up
with a very approximate superharmonic function $W(\xi)$. Let $i$ be the number of infected leaves,
and let $I(\xi)=1$  if the root is infected and $=0$ otherwise.
$$
W(\xi) = e^{-\lambda i/10} \left( 1 - I(\xi) \frac{ (e^{\lambda/10} - 1)}{\lambda} \right).
$$  
To make the connection change Pemantle's $n$ (the number of leaves) to our $k$ and note that
his birth rate $\lambda = \alpha/\sqrt{n}$. Pemantle has an interesting heuristic discussion on pages 2015--2016 
that explains why this form is reasonable. However the 10's that are thrown in to make it is easier to prove it is 
superharmonic ruin its accuracy. 

Here, following the approach in \cite{ChaDur}, we will reduce to a discrete time one dimensional chain,
we will only look at times when $j=1$. When the state is $(i,0)$ with $i>0$,
the next event will occur after exponential time with mean $1/(i\lambda + i)$. The
probability that it will be the reinfection of the center is $\lambda/(\lambda+1)$. 
The probability it will be the healing of a leaf is $1/(\lambda+1)$.
Thus, the number of leaf infections $N$ that will be lost while the center is healthy
has a shifted geometric distribution with success probability $\lambda/(\lambda+1)$, i.e.,
$$
 P(N=j) = \left( \frac{1}{\lambda+1}\right)^{j} \cdot \frac{\lambda}{\lambda+1}
 \quad\hbox{for $j\ge 0$}.
$$
Note that
$$
EN = \frac{\lambda+1}{\lambda} -1 = \frac{1}{\lambda}.
$$

The next step is to modify the chain so that the infection rate is 0 when the number
of infected leaves is at least
\beq
L =pk\quad\hbox{where}\quad p=\lambda/(1+2\lambda).
\label{Ldef}
\eeq
(The reader will see the reasons that underlie this choice later.) Note that for the modified chain the number of infected leaves is always $\le pk$
and the number of uninfected leaves is $\ge (1-p)k$. Thus if we look at the embedded discrete time process for 
the contact process on the star and only look at times when the center is infected,
the process dominates $Y_n$ where
\begin{center}
 \begin{tabular}{lc}
 jump & with prob \\
$ Y_n \to Y_n-1$ & $pk/D$ \\
$ Y_n \to \min\{Y_n+1,pk\}$  & $\lambda(1-p)k/D$ \\
$ Y_n \to Y_n-N$ & $1/D$
\end{tabular}
\end{center} 
Here $N$ is independent of $Y_n$ and the denominator
\beq
D = pk + \lambda(1-p)k + 1 \le k + \lambda k + 1 \le (2+\lambda)k.
\label{Dbds} 
\eeq

The fact that $Y_n$ is a reflecting random walk will simplify computations. We will use the process to lower bound survival times.
 Before the infection on the star graph goes extinct it will spend
most of its time near $pk$, (i) this does not lose much compared to the more accurate birth and death chain, which uses the 
actual number of infected leaves not just a bound, and (ii) we make only a small error when we return to continuous time 
by assuming that jumps happen at the maximum rate. In \cite{ChaDur} it is shown, see Lemma 2.2 on page 2339, that

\begin{lemma}
Suppose $\lambda \le 1$ and $\lambda^2 k \ge 50$. Let $L_0 = \lambda k/4$  and $S_0 =\frac{1}{4L_0} \exp(k\lambda^2/80)$. Then
$$
P_{L,i}\left( \inf_{t\le S_0} |\xi_t| \le 0.4 L_0 \right) \le 7 e^{-\lambda^2 L_0/80} \qquad\hbox{for $i=0,1$.}
$$
\end{lemma}

\noindent
In contrast our Lemma 2.4 will show that if $L=\lambda k/(1+2\lambda)$ and $b=\ep L$  
\beq
P_{L,1}\left( \inf_{t\le S} |\xi_t| \le bL \right) \le (3+\lambda)(1+\lambda/2)^{-\ep L}
\quad\hbox{where}\quad S=\frac{1}{(2+\lambda)2k}(1+\lambda/2)^{L(1-2\ep)}. 
\label{L24}
\eeq
Part of the improvement comes from simply replacing $L_0$ by $L$ and 0.4 by $\ep$, but the most 
important change is to construct a more accurate superharmonic function. If one is proving that a critical value is 0, as \cite{Pem92} and \cite{ChaDur} were, then it is not harmful to be off by a large constant factor, but if we are trying to get a good positive upper bound we need to be accurate.

In a companion paper we have shown that the improved lower bound is sharp. Let $T_{0,0}$ be the extinction time of the contact process on a star graph with $n$ leaves. We write $E_{i,j}$ for the expectation of the process starting from state $(i,j)$.

\mn
{\bf Lemma 4 in \cite{pertree}.} {\it Let $K = \lambda n/(n+1)$. For any $\epsilon>0$, the contact process on the star graph has 
\beq
E_{K,1} T_{0,0} \le (\log n) e^{(1+\epsilon)\lambda^2 n }.
\label{survub}
\eeq
when $n$ is sufficiently large.}

\noindent
If $\lambda^2n  \to\infty$ then the $\log n$ prefactor can be absorbed by changing $\ep$ however it is important if $\lambda = O(1/\sqrt{n})$, since in this case the exponential is $O(1)$.

In contrast, the lower bound time $T$ from \eqref{L24}, ignoring the prefactor, is
$$
(1+\lambda/2)^{L(1-2\ep)} \approx \exp\left( (1-2\ep) \frac{\lambda^2 k}{2(1+2\lambda)} \right).
$$
If $\lambda$ is small then the term in the exponential is about 1/2 the one in \eqref{survub}. Strictly speaking these results are not sharp (on the exponential scale) but a factor of 2 is much better than the factor of 80 that appears in \cite{ChaDur}. It is not clear which result gives the right answer. The result in \eqref{survub} is proved by looking at the first time the center becomes healthy and then all of the leaves become healthy before the center is reinfected. At first sight this bound seems crazy, but the calculations above show that it is fairly accurate. We have not been able to finding a good subharmonic function for $Y_n$ to find a better upper bound so we leave it to a clever reader to determine the nature of the large deviation event that wipes out the infection on the star.

\subsection{Galton-Watson trees}

Given an offspring distribution $p_k$, we construct a Galton-Watson trees as follows. Starting with the root, each individual has $k$ children with probability $p_k$. Pemantle has shown that 

\mn
{\bf Theorem 3.2 in \cite{Pem92}.} {\it  There are constants $c_2$ and $c_3$ so that if $\mu$ is the mean of the offspring distribution, 
then for any $k > 1$, if we let $r_k = \max\{ 2 , c_2 \log(1/kp_k)/\mu\}$ .}
\beq
\lambda_2 < c_3\sqrt{r_k \log r_k \log(k)/k}.
\label{Pemub}
\eeq

\mn
If the offspring distribution in the Galton-Watson tree is a stretched exponential 
$p_k = c_\gamma  \exp(-k^\gamma)$ with $\gamma<1$ then  $\log(1/kp_k) \sim k^\gamma$
and hence $ \lambda_2=0$.

Given this result, it is natural to ask about the critical values $\lambda_1$ and $\lambda_2$ when degrees have a geometric distribution. $p_k = (1-p)^{k-1} p$ for $k \ge 1$. The most interesting problem is to prove $\lambda_1 >0$. Here, we prove upper bounds. 

\begin{theorem} \label{ubglam1}
$\lambda_1 \le p/(1-p)$. 
\end{theorem}

\begin{proof} Modify the contact process so that births from a site can only occur on sites further from the root. Each vertex $x$ will be occupied at most once. If $x$ is occupied then it will give birth with probability $\lambda/(\lambda+1)$ onto each neighbor $y$. The birth events are not independent but that is not important. If we let $Z_n$ be the  number of sites at distance $n$ that are ever occupied, $Z_n$ is a branching process in which the offspring distribution has mean $\lambda /((\lambda+1)\cdot p)$ which is  $>1$ if $\lambda> p/(1-p)$.
\end{proof}

When $p_k = (1-p)^{k-1} p$, $\log(1/kp_k) \sim c_p k$, so \eqref{Pemub} gives a finite upper bound on $\lambda_2$. It is difficult to trace through all the calculations to get an explicit lower bound. However, Pemantle uses $e^{-1}/5 = 0.0735$ as the lower bound for the probability of long time survival starting with only the center of a large degree star graph occupied, while Lemma \ref{ignite} gives $1 - 3k^{-1/3}$ when the degree is $k$. 
This probability $e^{-1}/5$ appears cubed near the end of his proof, so we think that his bound is much worse than the following:

\begin{theorem} \label{ubglam2}
If $p_k = 2^{-k}$ for $k\ge 1$, then $\lambda_2 \le 2.5$.
\end{theorem} 

\noindent
This result is proved by combining our new estimates for the contact process on stars with the mysterious Lemma 2.4 in Pemantle's paper \cite{Pem92} (see Lemma \ref{magic} below).

The proof works for a general geometric $p_k = (1-p)^{k-1}p$, $k \ge 1$. We cannot get a nice formula for the upper bound 
as a function of $p$ but the upper bounds can easily be computed numerically and graphed.
These upper bounds are only interesting for small $p$. A Galton-Watson tree with $p_0=0$ and $p_1<1$ contains a copy of $\ZZ$ (start with a vertex with two children
and follow their descendants) so using Liggett's bound on $\lambda_c(\ZZ)$ proved in \cite{Ligbd} we conclude $\lambda_2 \le 2$ for all $0<p<1$.

\begin{figure}[tbp] 
  \centering
  \includegraphics[bb=53 58 737 555,width=4.09in,height=2.97in,keepaspectratio]{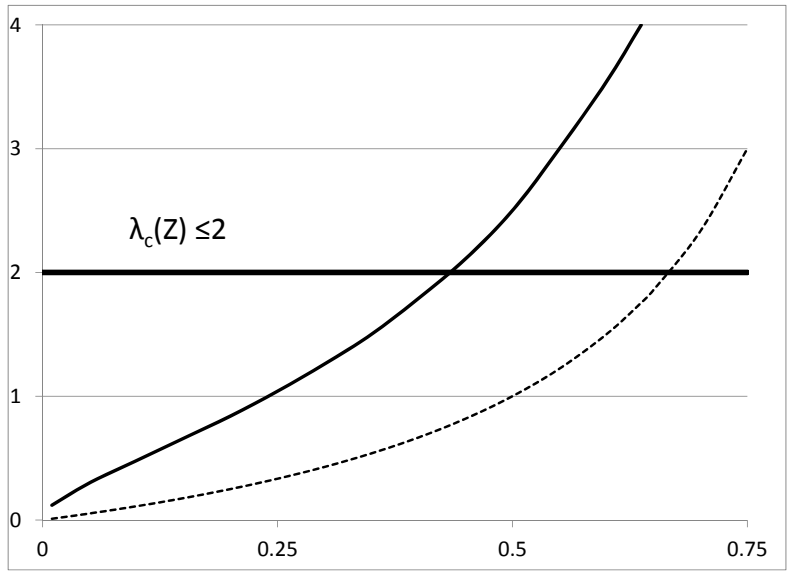}
  \caption{Upper bounds on $\lambda_2$ (solid line) and $\lambda_1$ (dotted line) as a function of $p$ for the geometric degree distribution. The graph is computed by using \eqref{suff}.}
  \label{fig:lambda2}
\end{figure}

In addition, the proof of Theorem \ref{ubglam2} yields an improvement of Pemantle's result for stretched exponential distributions. We say that
$p_k$ is subexponential if 
$$
\limsup_{k\to\infty} (1/k)\log p_k = 0.
$$

\begin{theorem} \label{subexp}
If the offspring distribution $p_k$ for a Galton-Watson tree is subexponential and has mean $\mu>1$ then $\lambda_2=0$.
\end{theorem}

\noindent
Note that $\lambda_2=0$ implies $\lambda_1=0$.

In the version of this paper submitted for publication in ALEA, we conjectured that the result in Theorem \ref{subexp} is sharp. This has recently been proved by 

\mn
{\bf  Bhamidi, Nam, Nguyen, and Sly \cite{Slyetal}} {\it Consider the contact process on the Galton-Watson tree with offspring distribution
$\zeta$, and suppose that only the root of the tree is initially infected. If $E(\exp(c\zeta))<\infty$
for some $c > 0$, then $\lambda_1>0$.}

\mn
They also prove results for random graphs. See \cite{Slyetal} for more details.

\subsection{Finite graphs}

Consider the contact process on $\{ -n, \ldots n \}$ starting from all sites occupied and let $\tau_n = \inf \{ t : \xi_t = \emptyset\}$.  Combining results of Durrett and Liu \cite{DLiu} and  Durrett and Schonmann \cite{DSch} gives the following results 

\mn
(i) If $\lambda < \lambda_c$ then there is a constant $\gamma_1(\lambda)$ so that 
$$
\tau_n/\log n \to \gamma_1(\lambda)\quad\hbox{in probability.}
$$
(ii) If $\lambda > \lambda_c$ then there is a constant $\gamma_2(\lambda)$ so that 
$$
(\log \tau_n)/n \to \gamma_2(\lambda)\quad\hbox{in probability.}
$$
(iii) When $\lambda > \lambda_c$  there is  ``metastability'':
$$
\tau_n/E\tau_n \Rightarrow \hbox{exponential}(1)
$$
where $\Rightarrow$ means convergence in distribution. Intuitively, the process on the interval stays exponentially long in a state that looks like the stationary distribution for the process on $\ZZ$, and then suddenly dies out.

Results on $\ZZ^d$ with $d>1$ had to wait for the work of Bezuidenhout and Grimmett \cite{BezGr}, who showed that in $d>1$ the contact process dies out at the critical value and in doing so introduced a block construction that can be used to study the supercritical process. Mountford \cite{Mountmet} proved the metastability result in 1993 and that $(\log \tau_n)/n^d \to \gamma(\lambda)$
in 1999, see \cite{Mountcrit}. 

Stacey \cite{Sta01} studied the contact process on a tree truncated at height $\ell$, $\TT^d_\ell$. To be precise, the root has degree $d$,
vertices at distance $0< k < \ell$ from the root have degree $d+1$, while those at distance $\ell$ have degree 1.
Cranston, Mountford, Mourrat, and Valesin  improved Stacey's result to establish that the time to extinction
starting from all sites occupied $\tau_\ell^d$ satisfies

\begin{theorem} \label{fintree}
\cite{CMMV} (a) For any $0< \lambda < \lambda_2(\TT^d)$ there is a $c \in (0,\infty)$ so that as $\ell\to\infty$
$$
\tau_\ell^d /\log|\TT^d_\ell| \to c \quad\hbox{in probability.}
$$
(b) For any $\lambda_2(\TT^d) < \lambda < \infty$ there is a $c \in (0,\infty)$ so that as $\ell\to\infty$
$$
\log(\tau_\ell^d) /|\TT^d_\ell| \to c \quad\hbox{in probability.}
$$
Moreover $\tau_\ell^d/E\tau_\ell^d$ converges to a mean one exponential.
\end{theorem}

When a tree is truncated at a finite distance, a positive fraction of the sites are on the boundary.
A more natural finite version of a tree is a random regular graph in which all vertices have degree $d+1$.
In this case there is no boundary and the graph has the same distribution viewed from any point. If there are $n$ vertices,  
the graph looks like $\TT_d$ in neighborhoods of a point that have $\le n^{1/3}$ vertices.
Mourrat and Valesin  have shown for a random regular graph,
the time to extinction starting from all sites occupied $\tau_n$ satisfies:

\begin{theorem} \label{randreg}
\cite{MouVal} (a) For any $0< \lambda < \lambda_1(\TT^d)$ there is a $C < \infty$ so that as $n\to\infty$
$$
P( \tau_n  < C \log n ) \to 1,
$$
(b) For any $\lambda_1(\TT^d) < \lambda < \infty$ there is a $c  > 0$ so that as $n\to\infty$
$$
P( \tau_n > e^{cn} ) \to c .
$$
\end{theorem}

\mn
Notice that the threshold in the second result comes at $\lambda_1$, while the one in Stacey's result comes at $\lambda_2$.
The difference is that when $\lambda \in (\lambda_1,\lambda_2)$ on the infinite tree the origin is in the middle of linearly growing vacant 
region. On the truncated tree the system dies out when the vacant region is large enough. However, on the random regular graph the 
occupied sites will later return to the origin. Durrett and Jung \cite{DJung} investigated the qualitative differences
between $\lambda \in (\lambda_1,\lambda_2)$ and $\lambda > \lambda_2$ on the small world graph.

To construct a random graph $G_n$ on the
vertex set $\{1, 2, \ldots, n\}$ having a specified degree distribution, we use the {\it configuration model}.
Let $d_1, \ldots, d_n$ be independent and have the distribution $P(d_i=k)=p_k$. 
In order to have a valid degree sequence, we condition
on the event $E_n=\{d_1+\cdots+d_n \text{ is even}\}$.  
Since $P(E_n)\rightarrow 1/2$ as $n \to\infty$,
the conditioning will have a little effect on the distribution of
$d_i$'s. Having chosen the degree sequence $(d_1, d_2, \ldots, d_n)$,
we attach $d_i$ half-edges to the vertex $i$, and then pair
these half-edges at random. This procedure may produce a graph with
self-loops or parallel edges, but we will ignore that problem for the moment. 

In the early 2000's physicists studied the contact process on a random graphs with a power-law degree distribution, i.e., the degree of each vertex is $k$ with probability 
$$
p_k \sim Ck^{-\alpha}\quad\hbox{as $k\to\infty$}.
$$ 
Pastor-Satorras and Vespignani \cite{PSV1,PSV2,PSV3} have made an
extensive study of this model using mean-field methods. Their nonrigorous
computations suggest the following conjectures about $\lambda_c$, the
threshold for ``prolonged persistence" of the contact process, and
the critical exponent $\beta$, that controls the rate at which the equilibrium density of occupied
sites $\rho(\lambda)$ goes to 0, i.e., $\rho(\lambda) \sim C (\lambda-\lambda_c)^{\beta}$.

 \begin{itemize}

 \item If $\alpha \le 3$, then $\lambda_c = 0$. If $\alpha<3$ then $\beta = 1/(3-\alpha)$.

 \item If $3 < \alpha \le 4$, then $\lambda_c > 0$ and $\beta= 1/(\alpha-3) >1$.

 \item If $\alpha > 4$, then $\lambda_c>0$ and $\beta=1$.

 \end{itemize}

\noindent
See also Section V of \cite{PSetc}. The values of $\beta$ quoted above are given in formula (29) of \cite{PSetc}.

Chatterjee and Durrett \cite{ChaDur} showed in 2009 that $\lambda_c>0$ is
not correct when $\alpha>3$ and $P( d_i \le 2 ) = 0$.
The last condition guarantees that the graph is connected and that random walks
on the graph have good mixing properties. They only proved survival for time $\exp(O(n^{1-\ep}))$ but they
obtained bounds on the critical exponent $\beta$.

In 2013 Mountford, Mourrat, Valesin, and Yao \cite{Mbeta} extended the results of \cite{ChaDur}
to include $2 < \alpha \le 3$ and proved upper and lower bounds that had the same dependence on
$\lambda$ but different constants, showing that
$$
\rho(\lambda) \sim
\begin{cases} \lambda^{1/(3-\alpha)} & 2 < a \le 5/2 \\
\lambda^{2\alpha-3} \log^{2-\alpha}(1/\lambda) & 5/2 < \alpha \le 3 \\
\lambda^{2\alpha-3} \log^{4-2\alpha}(1/\lambda) & 3 < \alpha
\end{cases}
$$
The result for $2< \alpha \le 5/2$ agrees with the mean-field 
calculations quoted above but that formula is claimed to hold for $2 < \alpha < 3$.
Figure 2 gives a visual comparison of the mean-field and rigorous resultls for critical exponents.
For more about why the change occurs at 5/2 see the next section and \cite{Mbeta}.
Three years later, Mountford, Mourrat, Valesin, and Yao \cite{Mexp} showed that
 for all $\lambda>0$, there is a $c(\lambda)>0$ so that the survival time $\ge e^{cn}$ with high probability.

\begin{figure}[tbp] 
  \centering
  \includegraphics[bb=53 58 738 555,width=4in,height=2.9in,keepaspectratio]{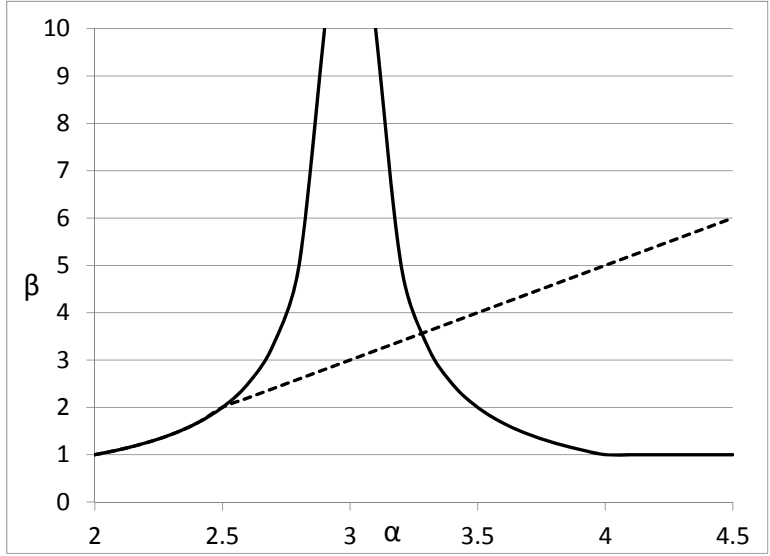}
  \caption{Mean field critical exponents (solid line) versus rigorous results (dashed line) as $\alpha$ varies from 2 to 4.5.}
  \label{fig:crexp}
\end{figure}

\subsection{Critical value asymptotics when $\lambda_c = 0$}

While the results cited above show that the mean-field calculations are not correct, physicists have never said they were wrong. Indeed, a 2010 paper Castellano and Pastor-Satorras \cite{CPS} claims they knew the right answer all along. ``Already in 2003, Wang et al \cite{Wangetal} argued that the SIS epidemic
threshold on any graph is set by the largest eigenvalue of the adjacency matrix, $\Lambda$
\beq
\lambda_c(n) = 1/\Lambda(n).''
\label{Wangcr}
\eeq
Two years earlier Pemantle snd Stacey \cite{PemSta} proved that $1/\Lambda(n)$ is the critical value 
of branching random walk on the graph. To be precise they showed in Lemma 3.1 that

\mn
{\bf Theorem.} {\it Let $G = (V,E)$ be a graph and let $M(v,2n)$ be the number of paths with $2n$ steps that begin and end at $v$. Let 
$$
M=\lim_{n\to\infty} M(v,2n)^{1/2n}=\sup_n  M(v,2n)^{1/2n}.
$$
The limit exists by supermultipicativity and is independent of $v$. The critical probability for local survival of
the branching random walk is given by $1/M$.}

\mn
However, it is far from obvious why this should also be the critical value for the contact process. For example on $\ZZ$, the critical value $\lambda$ for branching random walk is 1/2 while for the contact process $\lambda_c \approx 0.82$. 

The first question that needs to be addressed before \eqref{Wangcr} can become a theorem is the definition of $\lambda_c$. According to page 942 of the 2015 survey paper in {\it Reviews of Modern Physics} \cite{PSetc} ``Above the epidemic threshold, the activity must be {\it endemic}, so that the average time to absorption is $O(e^{cn})$.'' To make it clear that they wanted to insist on this standard we note that the discussion continued with 

\begin{quote}
``Chatterjee and Durrett proved that in graphs with power law degree distribution $E{T} > \exp(O(  N^{1-\delta}))$ for any $\delta>0$. This result pointed to a vanishing threshold but still left the possibility for nonendemic long-lived metastable states.''
\end{quote}   

\noindent
Survival for time $e^{cn}$ is certainly the gold standard for prolonged persistence, but following the footsteps of Ganesh, Masoulie, and Towsley \cite{GMT}, we will accept survival for time $\exp(O(n^\ep))$ for some $\ep>0$ as evidence that $\lambda> \lambda_c$.

The proofs of \eqref{Wangcr} in \cite{Wangetal} and \cite{Chak} do not provide a lower bound on survival time. They let $n\to\infty$ to obtain a nonlinear dynamical system (NLDS). To explain, note that if we let $p_{i,t}$ be the probability $i$ is infected at time $t$ and let $\zeta_{j,t}$ be the probability $j$ does not receive infection at time $t$  then
\begin{align*}
\zeta_{i,t} & = \prod_{j : j \sim i} ( 1- \beta p_{j,t-1}) \\
1 - p_{i,t} & = (1-p_{i,t-1})\zeta_{i,t} + \delta p_{i,t-1} \zeta_{i,t}
\end{align*}
Then they argue that if $\lambda > \Lambda^{-1}$ then one of the eigenvalues of the linearization of the NLDS around 0 is $> 1$, see the Appendix of \cite{Chak}. It is not clear what the last conclusion implies in terms of persistence. Wang et al \cite{Wangetal} use \eqref{Wangcr} to conclude that the critical value for the contact process on a star graph with $n$ leaves is $1/\sqrt{n}$. 

The results discussed in Section 1.1 show that the survival time on the star graph increase dramatically when $\lambda$ changes from $O(1/\sqrt{n})$ to $\gg 1/\sqrt{n}$. However, the claim that critical value on a star graph is $1/\Lambda(n)$ is contradicted by \eqref{survub} which shows that if $\lambda = \alpha/\sqrt{n}$ then for large $n$
$$
E_{K,1} T_{0,0} \le e^{2\alpha^2} \log n
$$ 
where $K = \lambda n/(\lambda + 1)$. It is not hard to show that the time needed to go from $n$ to $K$ is $O(\log n)$. Thus the survival time is $O(\log n)$ which is much smaller than the $O(e^{cn})$ that  \cite{PSetc} demands. Since the results in Section 1.1 show that the survival time is $\exp(O(\lambda^2 n))$, we would have to take $\lambda >0$ independent of $n$ for the contact process on the start to survive for this long.

Returning to the implications of \eqref{Wangcr} for the contact process, the maximum eigenvalue of the adjacency matrix of a random graph is trivially $\ge d_{max}^{1/2}$ (generated by paths going back and forth between a vertex with degree $d_{max}$ and its neighbors). Using results of Chung, Lu, and Vu \cite{CLV} for the maximum eigenvalue for random graphs the authors of \cite{CPS} concluded that the critical value for power law random graphs satisfies
$$
\lambda_c \sim \begin{cases} \langle d \rangle / \langle d^2 \rangle  & 2 < \alpha < 5/2 \\
1/\sqrt{d_{max}} & 5/2 < \alpha 
\end{cases}
$$
where $d_{max}$ is the maximum degree in the graph, and $\langle d \rangle$, $\langle d^2 \rangle$ are the average values of $d(x)$ and $d(x)^2$ for the graph. 
More concretely 
$$
\lambda_c(n) \sim \begin{cases} n^{(\alpha-3)/(\alpha-1)}  & 2 < \alpha < 5/2 \\
n^{-1/2(\alpha-1)} & 5/2 < \alpha 
\end{cases}
$$

Using our results we can prove an upper bound on $\lambda_c$ that supports this prediction when $\alpha>3$. Here $a=\alpha-1$.

\begin{theorem} \label{powerasy}
Suppose that the degree distribution has 
$$
P( d(x) \ge k ) = 3^ak^{-a}\quad\hbox{for $k \ge 3$}. 
$$
We assume $a > 2$ so that $Ed(x)^2 < \infty$. Let $\lambda = n^{-(1-2\eta)/2a}$ and $\eta>0$. If we start from all 1's then there is an $\ep>0$ so that the system survives for time $\exp(O(n^{\ep}))$ with high probability.
\end{theorem}

\noindent
Combining this result with the fact that $1/\Lambda$ gives the critical value for branching random walk and hence a lower bound on the critical value for the contact process we have 
\beq
\lambda_c(n) = n^{-(1+o(1))/2a}.
\label{powerlamc}
\eeq

Next we consider the stretched exponential 
$$
P( d(x) \ge k ) = \exp(-x^{1/b} + 3^{1/b})\quad\hbox{ for $k \ge 3$}. 
$$
where $b>1$. In this case, the maximum degree vertex on a graph with $n$ vertices  is $\sim \log^b n$, so the maximum eigenvalue $\Lambda \sim \log^{b/2} n$ and the formula in \eqref{Wangcr} predicts that $\lambda_c \approx \log^{-b/2} n$ but results of \cite{Slyetal} show that this cannot be correct for $b \le 1$. In that case the moment generating function of the degree distribution is finite for some positive $\theta$ so $\lambda_c(n)$ converges to a positive limit.

\begin{theorem} \label{seasy}
Suppose $\lambda_n = \log^{(1-\eta)(1-b)/2} n$. If we start from all 1's then for any $\ep>0$ the system survives for time $\exp(O(n^{1-\ep}))$ with high probability.
\end{theorem}

\noindent
We believe that the last result gives the right answer.

\begin{conj} 
Suppose $\lambda_n = \log^{-a/2} n$ where $a > b-1$. If we start from all 1's then for any $\ep>0$ the system dies out by time  $\exp(O(n^{\ep}))$ with high probability.
\end{conj}

The remainder of the paper is devoted to proofs. Section 2 gives our results for the star graph. Section 3 proves our results for Galton-Watson trees. Section 4 gives the asymptotics for $\lambda_c(n)$.

\clearp

\section{Results for the star graph } \label{sec:ublam2}

Recall from (1) that we set
$$
L=pk \hspace{1ex} \text{ where } \hspace{1ex} p=\lambda/(1+2\lambda).
$$
The definition of $Y_n$ is givn right after that formula.

\begin{lemma} \label{super}
Let $e^\theta = 1/(1+\lambda/2)$. If $k$ is large enough $e^{\theta Y_n}$ is a supermartingale while $Y_n \in (0,pk)$.
\end{lemma}

\begin{proof} We begin by noting that
\begin{align}
E(\exp(\theta Y_{n+1}) - \exp(\theta Y_n)|Y_n=y)  =e^{\theta y}( e^{\theta} - 1) \lambda (1-p)k/D &
\label{smeq}\\
 + e^{\theta y}(e^{-\theta} - 1 ) pk/D  + \frac{e^{\theta y}}{D} \left[ \sum_{j=0}^\infty \left( \frac{e^{-\theta}}{1+\lambda} \right)^j 
\left( \frac{\lambda}{1+\lambda} \right) -1 \right]& .
\nonumber
\end{align}
The term in square brackets is
$$
\frac{1}{1 - e^{-\theta}/(1+\lambda)} \cdot \frac{\lambda}{1+\lambda} - 1 
 = \frac{\lambda}{1+\lambda - e^{-\theta}} - 1 
 = \frac{e^{-\theta} - 1}{1+\lambda - e^{-\theta}} \ge 0.
$$
Note that $\theta<0$ so the last inequality implies that we must take $e^{-\theta} < 1+\lambda$.

The first two terms are 
$$\frac{e^{\theta y}k}{D} \left( (e^\theta-1)\lambda(1-p)+(e^{-\theta}-1)p\right),$$
so we begin by solving 
$$(e^\theta-1)\lambda(1-p)+(e^{-\theta}-1)p=0.$$ 

Rearranging and setting $x=e^\theta$ we want
$$
x^2\lambda(1-p) - [\lambda(1-p)+p] x + p = 0.
$$
Factoring we have
$$
(\lambda(1-p)x - p)(x-1) = 0.
$$
Since $p = \lambda/(1+2\lambda)$ the smaller root is
$$
\frac{p}{\lambda(1-p)} = \frac{\lambda/(1+2\lambda)}{\lambda (1+\lambda)/(1+2\lambda)} = \frac{1}{1+\lambda}.
$$

We let $e^\theta = 1/(1+\lambda/2) \in (1/(1+\lambda),1)$ so that there is a $\delta>0$ with
$$
e^{\theta} \lambda(1-p) +  e^{-\theta} p = [\lambda(1-p) + p] - \delta 
$$
and hence 
$$
(e^{\theta} - 1)\lambda(1-p)k +  (e^{-\theta} - 1)pk + \frac{e^{-\theta} - 1}{1+\lambda - e^{-\theta}}
= - \delta k + \frac{e^{-\theta} - 1}{1+\lambda - e^{-\theta}}.
$$
From this we see that if $k$ is large enough $e^{\theta Y_n}$ is a supermartingale while $Y_n \in (0,pk)$.
The reason we restricted $Y_n$ to $(0,pk)$ is that when $Y_n\leq pk$, the number of infected leaves tends to grow, which makes it possible to construct a supermartingale  $e^{\theta Y_n}$ with $\theta<0$. 
Note that when $Y_n$ is small the number of infected leaves may become 0 before the center is reinfected
but in this case the number of lost infections $N$ is truncated.

\end{proof}

Let $T_\ell^- = \inf\{ n : Y_n \le \ell\}$ and let $T_m^+ = \inf\{ n : Y_n \ge m \}$. We write $P_i$  for the law of the process $Y_n$ starting with $Y_0=i$.

\begin{lemma} \label{exit}
Let $a,b\in (0,L)$. If $b < a$ then
$$
P_{a}(T^-_{b} < T^+_L) \le (1+\lambda/2)^{b-a}.
$$
\end{lemma}

\begin{proof} To estimate the hitting probability let $\phi(x) = \exp(\theta x)$ where we take $e^\theta = 1/(1+\lambda/2)$
and note that if $\tau = T^-_{b} \wedge  T^+_L$ then $\phi(Y(t\wedge \tau))$ is a supermartingale. Let $q=P_{a}(T^-_{b} < T^+_L)$. Using the optional stopping theorem we have
$$
q \phi(Y(T_b^-))+(1-q)\phi (Y(T^+_L)) \le \phi(a).$$

It is possible that $Y(T_b^-)<b$. Note that since $\theta < 0$, we have $\phi(x) \ge \phi(b)$ for $x \le b$. Hence, 
$$
q \phi(b) + (1-q) \phi(L) \le  \phi(a).
$$
Dropping the second term on the left, $q\le \phi(a)/\phi(b)= (1+\lambda/2)^{b-a}$
, which completes the proof.
\end{proof} 

\begin{lemma} \label{return}
If $R_L = \inf\{ n > T^-_{L-1} : Y_n = L \}$ and $b\in[0,L)$ then for sufficiently large $k$
$$
P_L(T_{b}^- < R_L) \le (2+\lambda)(1+\lambda/2)^{b-L}.
$$
\end{lemma}

\noindent
{\bf Remark.} Here, and in later lemmas, the computation of explicit constants is somewhat annoying. However, when we consider asymptotics for critical values,  $\lambda$ will go to 0, so we will need to know how the constants depend on $\lambda$. 

\begin{proof} To compute the left-hand side we break things down according to the first jump. The definition of $R_L$ allows us to ignore the attempted upward jumps that do nothing. Recall that $L=pk$. The jump is to $L-1$ with probability $pk/(pk+1)$ and to $L-j$ with probability $ \frac{ \lambda}{(1+\lambda)^{j+1}} \cdot \frac{1}{1+pk}$. In the first case the probability of going below $b$ before returning to $L$ is 
$$
\le (1+\lambda/2)^{b-(L-1)} = (1+\lambda/2) \cdot (1+\lambda/2)^{b-L}.
$$ 
In the second case we have to sum over the possible values of $L-j$. Using Lemma \ref{exit}
\begin{align*}
 & \le (1+\lambda/2)^{b-L} \sum_{j=1}^\infty \frac{\lambda}{(1+\lambda)^{j+1}} (1+\lambda/2)^j + \frac{\lambda}{1+\lambda}P_L(T_{b}^- < R_L)  \\
& \le (1+\lambda/2)^{b-L} \frac{\lambda}{\lambda+1} \cdot \sum_{j=0}^\infty \left(\frac{1+\lambda/2}{1+\lambda}\right)^j + \frac{\lambda}{1+\lambda}P_L(T_{b}^- < R_L)\\
&= 2(1+\lambda/2)^{b-L}+ \frac{\lambda}{1+\lambda}P_L(T_{b}^- < R_L).
\end{align*}
Noting that $\max\{ 2, 1+\lambda/2\} \le 2(1+\lambda/2)-\delta$ for some small $\delta<\lambda$, we have the following relation,
$$
P_L(T_{b}^- < R_L) \leq \frac{\lambda}{(1+\lambda)(1+pk)}P_L(T_{b}^- < R_L)+(2+\lambda-\delta)(1+\lambda/2)^{b-L}.
$$
Hence for $k$ sufficiently large, we have 
$P_L(T_{b}^- < R_L) \leq (2+\lambda)(1+\lambda/2)^{b-L}.$
\end{proof}

Recall that $\xi_t$ denotes the original contact process on the star graph with $k$ leaves.
\begin{lemma}\label{life}
Let $b=\ep L$ and $S=\frac{1}{(2+\lambda)2k}(1+\lambda/2)^{L(1-2\ep)}$
$$
P_{L,1}\left( \inf_{t\le S} |\xi_t| \le b \right) \le (3+\lambda)(1+\lambda/2)^{-L\ep}.
$$
\end{lemma}

\noindent
We have returned to unmodified process so $(L,1)$ means $L$ leaves are infected and the center is as well. Again when we write the state as a subscript we drop the parentheses.  

\begin{proof} 
Let $M=(1+\lambda/2)^{L(1-2\ep)}$. By Lemma \ref{return} the probability that the chain fails to return $M$ times to $L$ before going below $\ep L$ is
 $$
 \le (2+\lambda)(1+\lambda/2)^{-L\ep}.
 $$
Using Chebyshev's inequality on the sum $S_M$ of $M$ exponentials with mean 1 (and hence variance 1),
 $$
 P( S_M < M/2 ) \le 4/M.
 $$
When the number of infected leaves is $\le L$ maximum jump rate is $D \le (2+\lambda)k$ so
$$
P \left( \frac{S_M}{(2+\lambda)k} \le \frac{(1+\lambda/2)^{L(1-2\ep)}}{2(2+\lambda) k }\right)
\le 4(1+\lambda/2)^{-L(1-2\ep)} \le (1+\lambda/2)^{-L\ep} .
$$
for large $L$. Adding up the error probabilities gives
$$
P_{L,1}\left( \inf_{t\le S} |\xi_t| \le b \right)\leq (3+\lambda)(1+\lambda/2)^{-L\ep}
$$
and completes the proof.
\end{proof}
 
Up to this point we have shown that if a star has $L$ infected leaves it will remain infected for a long time. To make this useful, we need estimates about what happens when the star starts with only the center infected.  Let $T_{0,0}$ be the first time the star is healthy. We use the pair $(n,i)$ to denote the state of the star graph, where $n$ is the number of infected leaves and $i$ indicates the state of the center ($i=1$ means the center is infected).

\begin{lemma} \label{ignite}
Let $\lambda>0$ be fixed and $K = \lambda k^{1/3}$.  Then for large $k$
\begin{align*}
& P_{0,1}( T^+_K > T_{0,0} )  \le  2\lambda k^{-1/3},  \\
& P_{K,1} (T_{0,0} < T^+_L)  \le  k^{-1/3},\\
& E_{0,1} (T^+_L | T^+_L < T_{0,0})  \le  2/\lambda.
\end{align*}
\end{lemma}

\begin{proof} 
Clearly 
$$
P_{0,1}(T^+_K < T_{0,0}) \ge \prod_{j=0}^{K-1} \frac{(k-j)\lambda}{1+(k-j)\lambda+j}
$$
so subtracting the last inequality from $1 = \prod_{j=0}^{K-1} 1$ and using Lemma 3.4.3 from  \cite{PTE}
$$
P_{0,1}(T^+_K > T_{0,0}) \le \sum_{j=0}^{K-1} \frac{1+j}{(k-j)\lambda}  \le \frac{\lambda^2 k^{2/3}}{(k-\lambda k^{1/3}) \lambda } \le 2\lambda k^{-1/3}.
$$
For the second result we use the supermartingale $e^{\theta Y_n}$ from Lemma \ref{super}. If $q=P_{K,1}(T_{0,0} < T^+_L)$, using optional stopping theorem we have
$$
q\cdot 1 + (1-q)e^{\theta L} \le e^{\theta K}.
$$
Dropping the second term on the left,
$$q\leq e^{\theta K}=(1+\lambda/2)^{-K} \le k^{-1/3}.$$

To bound the time we return to continuous time
\begin{center}
 \begin{tabular}{lc}
 jump & at rate \\
$ Y_t \to Y_t-1$ & $pk$ \\
$ Y_t \to \min\{Y_t+1,pk\}$  & $\lambda(1-p)k$ \\
$ Y_t \to Y_t-N$ & $1$
\end{tabular}
\end{center} 
Before time $V_L = T_{0,0} \wedge T^+_L$ the drift of $Y_t$ is at least
\beq
\mu = \lambda(1-p)k -pk -1/\lambda=\lambda pk -1/\lambda
\label{Ydrift}
\eeq
so $Y_t - \mu t$ is a submartingale. Stopping this martingale at the bounded stopping time $V_L \wedge t$ 
$$
EY(V_L\wedge t) - \mu E(V_L\wedge t) \ge EY_0 \ge 0.
$$
Since $EY(V_L\wedge t) \le L$, it follows that
$$
E(V_L\wedge t)\le \frac{L}{\mu} = \frac{pk}{\lambda pk - 1/\lambda},
$$ 
where $p = \lambda/(1+2\lambda)$, so if $\lambda$ is fixed and $k$ is large
$$
E(V_L \wedge t) \le 2/\lambda,
$$
which completes the proof. 
\end{proof}

Combining Lemmas \ref{life} and \ref{ignite} gives the following. When $G$ occurs, we say the star at 0 is good.

\begin{lemma} \label{good}
Let $A_t$ denote the number of infected leaves at time $t$ and take $S$ as in Lemma \ref{life}. 
Define $G = \{ \inf_{k^{2/3} \le t\le S} |A_t| \ge \ep L \}$. If $\lambda>0$ is fixed and $k$ is large then
\beq
P_{0,1}\left( G \right) \ge 1 - (2+2\lambda) k^{-1/3}
\label{from0}
\eeq
\end{lemma}

\begin{proof} Lemma \ref{ignite} implies
\begin{align*}
P_{0,1} (T^+_L<k^{2/3})&\geq P_{0,1}(T^+_L<k^{2/3}|T^+_L < T_{0,0})P_{0,1}(T^+_L < T_{0,0}) \\
&\geq (1-(2/\lambda)k^{-2/3})(1 - (1+2\lambda) k^{-1/3})\geq 1-(\frac{3}{2}+2\lambda)k^{-1/3}
\end{align*}
for large $k$. By the definition in \eqref{Ldef}, $L=pk$ where $p=\lambda/(1+2\lambda)$. Lemma \ref{life} tells us that
$$
P_{L,1}\left( \inf_{t\le S} |\xi_t| \le \ep L \right) \le (3+\lambda)(1+\lambda/2)^{-L\ep}.
$$
Since $\lambda$ is fixed the right-hand side is $\le k^{-1/3}/2$ for large $k$. Adding up the error probabilities completes the proof.
\end{proof}

\clearp

\section{Proofs of results for Galton-Watson trees} \label{sec:GWtrees}

In the previous section we developed estimates for the contact process on stars. The next step is to obtain estimates on the 
probability of ``pushing an infection from one star to another.'' When $\lambda>0$ is fixed we have to be careful not to lose too much. 

 \begin{lemma} \label{transfer}
 Let $v_0, v_1, \ldots v_r$ be a path in a graph and suppose that $v_0$ is infected at time 0. Then there is a $\gamma>0$ so that the probability that $v_r$ will become infected by time $2r$ is
 $$
\ge \left(\frac{\lambda}{\lambda+1} \right)^r (1 - \exp(-\gamma r)).
$$
If $\ep>0$ and we let $\hat\lambda = (1-\ep) \lambda/(\lambda+1)$ then for large $r$ this probability is 
$\ge \hat \lambda^r.$
\end{lemma}

\begin{proof}

The probability that $v_{i-1}$ infects $v_{i}$ before it is cured is $\lambda/(1+\lambda)$. When this transfer of infection occurs the
amount of time is $t_i$ exponential with rate $1+\lambda$. By large deviations for the exponential distribution 
$P( t_1 + \cdots + t_r > 2r ) \le e^{-\gamma r}$ for some $\gamma > 0$. 
 \end{proof}

We say a star is \textit{nice} if starting from $L$ infected leaves, the event in Lemma \ref{life} occurs. Recall that $S=\frac{1}{(2+\lambda)2k}(1+\lambda/2)^{L(1-2\ep)}$ as in Lemma \ref{life}.

\begin{lemma} \label{infect}
Run the contact process on a graph consisting of a star of size $k$ to which there has been added a single chain 
$v_1, \ldots v_r$ of length $r$ where $v_1$ is a neighbor of $0$, the center of the star. Suppose that at time 0 there are $L$ infected leaves and the star with center 0 is nice. 

For large $r$ and $k$, the probability that $v_r$ will not be infected before time $T= m (2r+1)$ for some $m\leq S/(2r+1)$ is
$$
\le (1-\hat\lambda^r)^{m}.
$$
\end{lemma}

\begin{proof}
Consider a sequence of times $t_i = (2r+1)i$ for $i\ge 1$. The center $0$ may not be infected at time $t_i$ but since the number of infected neighbors is $\ge \ep L$ the center will be infected by time $t_i+1$ with probability at least $1-e^{-\lambda \ep L}$.
By Lemma \ref{transfer} the probability $v_r$ is successfully infected in $[t_i,t_{i+1})$ 
is 
$$\ge (1-e^{-\lambda \ep L})\left(\frac{\lambda}{\lambda+1} \right)^r (1 - \exp(-\gamma r))\geq \hat{\lambda}^r$$ 
for sufficiently large $r$ and $k$. The desired result follows.
\end{proof}

\mn
{\bf Remark.} Due to the way the proof is done, if we condition on 0 being good then successes on two different
chains are independent events.

\medskip
To prepare for the proof of the main results we need the next lemma, which is Lemma 2.4 from \cite{Pem92}.
Let $\varphi(x)=\sum_{n=0}^\infty p_n x^n$ be the generating function of the Galton-Watson tree. We will apply Lemma \ref{magic} to
$$
f(t) = P(0 \in \xi^0_t ) \ge p_k P( 0 \in \xi^0_t | \hbox{ 0 has at least $k$ children}).
$$

\begin{lemma} \label{magic}
Let $H$ be any nondecreasing function on the nonnegative reals with $H(x) \ge x$ when $x \in[0,x_0]$.
If $f$ satisfies $(i)$ $\inf_{0\le t \le L} f(t) > 0$ and $(ii)$ $f(t) \ge H(\inf_{0\le s \le t-L} f(s))$ for $t \ge L$ some $L>0$ then
$\liminf_{t\to\infty} f(t) > 0$.
\end{lemma}

\begin{proof}
For any $t_0$ and $\ep>0$, (ii) implies that there is a decreasing sequence $t_i$ with $t_{i+1} \le t_i - L$ and $t_k < L$ for some $k$
$$
f(t_i) \ge H(f(t_{i+1})) - \ep 2^{-i}.
$$
If $f(t_i) < x_0$ for all $1\le i \le k$ then
$$
f(t_i) \ge f(t_{i+1}) - \ep 2^{-i}
$$
and summing gives $f(t_0) > f(t_k) - \ep$ which gives the desired result. Suppose now that $j$ is the smallest index with $f(t_j) > x_0$.
If $j=0$ we have $f(t_0)> x_0$. If $j=1$ we have $f(t_0) \ge H(x_0)$. If $j \ge 2$ we have
$$
f(t_0) \ge f(t_{j-1}) - \ep \ge H(x_0)-\ep
$$
so in all cases we get the desired conclusion. 
\end{proof}

\mn
{\bf Proof for $p_n = 2^{-n}$, $n \ge 1$.} Our proof follows the outline of the proof of Theorem 3.2 in \cite{Pem92}, see pages 2109--2110. We can suppose without loss of generality that the root has degree $k$. Otherwise examine the children of the root until we find one with degree $k$ and apply the argument to the children of this vertex. There are two steps in the proof.

\begin{enumerate}
  \item Push the infection to vertices at a distance $r=k$ that have degree $k$.
  \item Bring the infection back to the root at time $t$ using Lemma \ref{magic}.
\end{enumerate}  

\mn
{\bf Step 1.} 
The mean of the offspring distribution 2. Let $Z_r$ be the number of vertices
at distance $r$ from $0$ and let $v^1_r, \ldots v^J_r$ be the subset of those that have exactly $k$ children, where $J$ is a random variable that represents the number of such vertices.

Since the root has degree $k$  and $p_k = 2^{-k}$ if we set $r=k$
$$
EJ \ge k \mu^{r-1} p_k = k/2,
$$
where $\mu=2$ is the mean offspring number.
 
If we condition on the value of $W = Z_r/(k\mu^{r-1})$ and let $\bar J = (J|W)$ be the conditional distribution of $J$ given $W$
then 
$$
\bar J = \hbox{Binomial}(k2^{r-1}W,2^{-k}).
$$
Let $M$ be the random number of vertices among  $v^1_r, \ldots v^J_r$ that are infected before time 
$$
S=\frac{1}{2k(2+\lambda)}(1+\lambda/2)^{L(1-2\epsilon)}
$$ 
defined in Lemma \ref{life}. The event $G = \{ \inf_{k^{2/3} \le t\le S} |A_t| \ge \ep L \}$ in Lemma \ref{good} occurs with high probability. By Lemma \ref{infect}, conditioning on $G$ the probability a given vertex will not become infected by time $S$ is
\begin{align*}
& p_{noi} \le (1 -\hat\lambda^k)^m \quad\hbox{where}\quad \hat \lambda = (1-\ep)\frac{\lambda}{\lambda+1}  \quad\hbox{and}\\
& m = \frac{S-k^{2/3}}{2k+1} \geq \frac{  (1+\lambda/2)^{L(1-2\ep)}   }{4k(2k+1)(2+\lambda)} 
\quad\hbox{with}\quad L = \frac{\lambda k}{1+2\lambda}.
\end{align*}
Combining the definitions and using $(1-x) \le e^{-x}$ we have
$$
p_{noi} \le \exp\left( - \frac{\Gamma^k}{4k(2k+1)(2+\lambda)} \right) \quad\hbox{where}
\quad \Gamma = \hat\lambda (1+\lambda/2)^{(1-2\ep)\lambda /(1+2\lambda)}.
$$
When $\lambda = 2.5$
\beq
 \frac{\lambda}{\lambda+1} (1+\lambda/2)^{\lambda/(1+2\lambda)} = 1.0014 > 1,
\label{super}
\eeq
so $\Gamma >1$ when $\ep$ is small and $p_{noi} \to 0$ as $k\to\infty$. From this we see that if $\delta>0$ then for large $k$
$$
EM \ge (1-\delta)EJ.
$$
The remark after Lemma \ref{infect} implies that if we condition on the value of $W$ and let $\bar M = (M|W)$
then 
$$
\bar M \ge \hbox{Binomial}(k2^{r-1}W,2^{-k}(1-\delta)).
$$
To prepare for the following two generalizations of the result for Geometric(1/2) offspring distribution we ask the reader to verify that in Step 2, all we use is the
fact that \eqref{super} implies the bounds on $EM$ and $\bar M$.

\mn
{\bf Step 2.} Let $H_1(t)=P(v^i_r \in \xi_{t-S} \text{ for some $1\leq i\leq J$})$ and 
$$
H_2(t)=P(0\in \xi_t | v^i_r\in \xi_{t-S} \text{ for some $1\leq i\leq J$}),
$$
so that $f(t)\geq H_1(t)H_2(t)$. Fix $t>2S$ and let 
$$
\chi(t)= \inf\{f(s) : s \le  t - S \}.
$$
Since $t$ is fixed, we  simplify the notation and write $\chi(t)$ as $\chi$.

Ignore all but the first infection of each $v^i_r$ by its parent. any of these will evolve independently from the time $s<S$ it is first infected, and will be infected at time $t-S$ with probability at least $\chi$. Thus given $M$ the number of infected at time  $t - S$ will dominate $N = \hbox{binomial}(M,\chi)$. If we let $\bar N = \hbox{binomial}(\bar M,\chi)$ and let $\delta>0$,  then by Lemma 2.3 in \cite{Pem92} we see that there exists a $\varepsilon>0$ such that
$$
P( \bar{N} \ge 1 ) \ge (1-\delta)\chi EM \wedge \varepsilon
$$
Therefore $H_1(t)\geq  (1-\delta)\chi EM \wedge \varepsilon$ when $t>2S$.

Finally, if some $v^i_r$ is infected at time  $t - S$ then the probability of finding $0$ infected
at time $t$ is bounded below by $\rho_1\rho_2$ where

\begin{itemize}
  \item 
$\rho_1$ is the probability that the contact process starting with only $v^i_r$ infected at time 
$t - S$ infects $0$ at some time $s$ with $t - S \le s \le t$. By Lemmas \ref{ignite}, \ref{good}, and \ref{infect}, $\rho_1 \ge 1-\delta$. 

\item
$\rho_2$ is the probability $0$ is infected at time $t$ given the infection of $0$ at such a time $s$. For any $\ep>0$, by Lemma \ref{infect} the probability that $0$ have not been infected by time $S/2$ is less than $\ep$ when $k$ is sufficiently large. By Lemma \ref{good}, with probability $\geq 1-(2+2\lambda)k^{-1/3}$ there should be at least $\ep L$ infected leaves at time $t-\ep$. Hence 0 is infected at $t$ with probability at least $(1-e^{-\lambda \ep^2 L})e^{-\ep}$, where the second term guarantees that the root is infected at time $t$. Choosing $\ep$ is sufficiently small and $k$ sufficiently large gives $\rho_2\geq 1-\delta$.

\end{itemize}

\noindent
Thus
$$
f(t) \geq 
\begin{cases}
\chi(t)EM(1-\delta)^3 \wedge \varepsilon & t>2S,\\
\inf_{0\leq s \leq 2S} f(s) & S \leq t\leq 2S.\\
\end{cases}
$$

We can take $\varepsilon < \inf_{0\leq s \leq 2S} f(s)$ so that $f(t)\geq \chi(t)EM(1-\delta)^3 \wedge \varepsilon$ for all $t\geq S$. The result now follows from  Lemma \ref{magic} with $L=S$ and $H(x)=(1-\delta)^3(EM) x \wedge \varepsilon$.

\mn
{\bf Proof for $p_n = (1-p)^{n-1} p$.} It is now straightforward to replace 1/2 by $p$. We only have to pick $k$ and $r$ so that we can prove the analogue of \eqref{super}.
The mean of the offspring distribution is $1/p$. Let $Z_r$ be the number of vertices
at distance $r$ from $0$ and let $v^1_r, \ldots v^J_r$ be those that have exactly $k$ children.
Since the root has degree $k$  and $p_k = (1-p)^{k-1} p$ 
\beq\label{meanoff}
EJ \ge k (1/p)^{r-1} (1-p)^{k-1} p.
\eeq
In this case we want to pick $r$ so that $(1/p)^r(1-p)^k \approx 1$. Hence $EJ$ can be large when $k$ is large. Ignoring the fact that $r$ and $k$ must be integers this means
$$
r/k = \log(1-p)/\log p.
$$

Let $M$ be the random number of vertices among  $v^1_r, \ldots v^J_r$ that are infected before time $S$. 
By Lemma \ref{infect} the probability a given vertex will not become infected is
$$
\le (1 -\hat\lambda^r)^{\lceil S/(2r+1) \rceil} \le \exp\left(-\frac{\Gamma^k}{2k(2r+1)(2+\lambda)} \right)
$$
where $\Gamma=\hat\lambda^{r/k} (1+\lambda/2)^{(1-2\ep)\lambda /(1+2\lambda)}$.
That is, if we choose $\lambda$ such that 
\beq
\left( \frac{\lambda}{\lambda+1} \right)^{r/k} \cdot (1+\lambda/2)^{\lambda/(1+2\lambda)} > 1
\label{suff}
\eeq
then we have $\Gamma>1$ for large $k$. By the same reasoning as before this choice of $\lambda$ gives an upper bound on $\lambda_2$.

If we want to graph the bound as a function of $p$ it is better to work backwards. Given $\lambda$ the second factor is
$>1$ so we can easily find the value of $r/k$ that makes this 1. Having done this we can easily compute the value of $p$
for which $\lambda$ gives the upper bound on $\lambda_2$.

\mn
{\bf Proof for subexponential distributions.} We suppose that the mean of the offspring distribution is $\mu>1$.
If $p_k$ is subexponential, i.e.,
$$\limsup_{k\to\infty} (1/k)\log p_k=0,$$
then for any $\delta$ there is a $k$ with $p_k \ge (1-\delta)^k$. It follows from the same reasoning as in (\ref{meanoff}) that we can take $r$ such that
$$
\frac{r}{k} = - \frac{\log(1-\delta)}{\log \mu}.
$$
Given any $\lambda>0$, \eqref{suff} will hold if $\delta$ is small enough, which implies local survival of the process. Therefore $\lambda_2=0$.

\clearp 

\section{Asymptotics for $\lambda_c$}

We begin with some general computations and then consider our two examples: power laws and stretched exponential. 

\mn
{\bf Survival on star graph.} Our first step is to adapt Lemma \ref{life} to the situation in which $\lambda \to 0$. For reasons that will become
clear when we prove Lemma \ref{ignite2} we have to modify the definition of $p$:
$$
p = (1-\ep) \frac{\lambda}{1+\lambda}, \qquad L = pk, \quad\hbox{and}\quad b = \ep L.
$$
Defining $Y_n$ as before

\begin{lemma} \label{super2}
Let $\ep>0$. If $\lambda/(1+2\lambda) < \ep$ then $(1+\lambda/2)^{-Y_n}$ is a supermartingale.
\end{lemma}

\begin{proof} $(1-p) = (1+\lambda\ep)/(\lambda+1)$ so we have
$$
\frac{p}{\lambda(1-p)} = \frac{1-\ep}{1+\lambda\ep}.
$$
The right-hand side is $< 1/(1+\lambda)$ when 
$$
1 + \lambda - \ep - \ep\lambda < 1 + \lambda \ep,
$$ 
which holds if $\lambda/(1+2\lambda)< \ep$, so the desired result follows from the proof of Lemma \ref{super}.
\end{proof}

\begin{lemma} \label{life2}
Let $\ep>0$ be fixed  $T = \exp( (1-4\ep)\lambda^2 k / 4 )$. If $\lambda$ is small then for large $k$ 
$$
P_{L,1} \left( \inf_{t \le T} |\xi_t| \le b \right) \le 4 \exp( - (1-3\ep) \lambda^2 k / 4) ,
$$
\end{lemma}

\begin{proof} It follows from Lemma \ref{life} that if $S=(1/2k(2+\lambda))(1+\lambda/2)^{L(1-2\ep)}$ then
$$
P_{L,1} \left( \inf_{t \le S} |\xi_t| \le b \right) \le (3+\lambda) (1+\lambda/2)^{-L(1-2\ep)},
$$
but now
$$
(1-\ep)L = (1-\ep)^2\lambda k/(\lambda+1) > (1-2\ep) \lambda k/(\lambda + 1).
$$
Expanding $\log(1+x) = x - x^2/2 + x^3/3 - \ldots$ and noting that if $x< 1$ then the right-hand side is an alternating series with decreasing terms
\begin{align*}
( 1 + \lambda/2 )^{-(1-2\ep)\lambda k/(1+\lambda)} & = \exp\left(  - (1-2\ep)\frac{\lambda k}{1+\lambda} \log(1+\lambda/2) \right)\\
& \le  \exp\left( - (1-2\ep)\frac{\lambda k}{1+\lambda} \left[ \frac{\lambda}{2} - \frac{\lambda^2}{8} \right] \right) \\ 
& \le \exp( -(1-3\ep)\lambda^2 k/4 ),
\end{align*}
when $\lambda$ is small. To convert the formula for $S$ we note that
\begin{align*}
(1/2k(2+\lambda))(1+\lambda/2)^{L(1-2\ep)} 
& = \frac{1}{2k(2+\lambda)} \exp\left( \frac{(1-\ep)(1-2\ep) \lambda k}{(\lambda + 1)} \cdot \log(1+\lambda/2) \right) \\
& \ge \frac{1}{6 k} \exp\left( \frac{(1-3\ep) \lambda k}{(\lambda + 1)} \cdot  \left[ \frac{\lambda}{2} - \frac{\lambda^2}{8} \right]  \right) \\
&\ge  \exp( (1-4\ep) \lambda^2 k/4),
\end{align*}
when $\lambda$ is small, which completes the proof. 
\end{proof}

\mn
{\bf Push.} Now we work with the configuration model. Let $p_k = P( d(x)=k)$ and suppose that 

\medskip
(i) $\sum_{k} k^2 p_k < \infty$,

(ii) $P(d(x)=k)=0$ for $k \le 2$. 

\mn
The first assumption implies that the size biased degree distribution $q_{j-1} = j p_j/Ed(x)$ has finite mean $\nu$. The second implies that the diameter of our graph $\sim (\log n)/\log \nu$.See Lemma 3.4.1 in \cite{RGD}. Hence Lemma \ref{transfer} implies that if $v_0, \ldots v_r$ is a path in the graph and $v_0$ is infected at time 0, then the probability $v_r$ will be infected by time $2r$ is, for large $r$, $\ge (\lambda/2)^r$. Let
$$
\kappa = n^{3\nu\log(2/\lambda)}.
$$
If $n$ is large then the distance between any two vertices is $\le 2\nu \log n$ with high probability. Thus the probability that one star can transfer its infection to another before time $2r\kappa $ is
\begin{align}
&\ge 1 - \left( 1- (\lambda/2)^{2\nu \log n} \right)^{\kappa} 
 = 1 - \left( 1- n^{-2\nu \log(2/\lambda)} \right)^{\kappa} 
\nonumber\\
& \ge 1 - \exp(  - n^{\nu \log(2/\lambda)} ) .
\label{xfer}
\end{align}

\mn
{\bf Ignition on star graph.} We have more work to do this time. 
The proof of Lemma \ref{ignite} requires that $K = \lambda k^{2/3} \to \infty$, and we need the new definition of $L$ in part (iii).

Recall that $T_m^+ = \inf\{ n : Y_n \ge m \}$.

\begin{lemma} \label{ignite2}
Let $K = \lambda k/\sqrt{\log k}$. If $\lambda \to 0$ and $\lambda^2 k \to \infty$ then for large $k$
\begin{align*}
& (i) P_{0,1}( T_K^+ > T_{0,0} )  \le  5/\sqrt{\log k},  \\
& (ii) P_{K,1} (T_{0,0} < T_L^+)  \le   \exp(- \lambda^2 k/2\sqrt{\log k} ), \\
& (iii) E_{0,1} \min\{ T_{0,0},  T_L^+ \} \le 2/\ep.
\end{align*}
\end{lemma}

\begin{proof}  
Let $p_0(t)$ be the probability a leaf is infected at time $t$ when there are no infected leaves at time 0 and
the central vertex has been infected for all $s \le t$. $p_0(0)=0$ and
$$
\frac{dp_0(t)}{dt} = - p_0(t) + \lambda(1-p_0(t)) = \lambda -(\lambda+1) p_0(t). 
$$
Solving gives 
$$
p_0(t) = \frac{\lambda}{\lambda + 1} ( 1 - e^{-(\lambda +1)t} ).
$$
As $t \to 0$
$$
\frac{ 1 - e^{-(\lambda +1)t} }{ (\lambda +1)t } \to 1,
$$
so if $t$ is small $p_0(t) \ge \lambda t/2$.
 
Taking $t = 4/\sqrt{\log k}$ it follows that the number of infected leaves at time $t$ dominates $B =  \hbox{Binomial}(k,2\lambda/\sqrt{\log k})$
$$
P_{0,1}( T_K^+ < T_{0,0} ) \ge P( B > K) \exp(-4/\sqrt{\log k}).
$$
The second factor is the probability that the center stays infected until time $4/\sqrt{\log k}$, and 
$$
\exp(-4/\sqrt{\log k})\ge 1 - 4/\sqrt{\log k}. 
$$
$B$ has mean $2\lambda k/\sqrt{\log k}$ and variance $\le 2\lambda k/\sqrt{\log k}$ so Chebyshev's inequality implies
$$
P( B < \lambda k/\sqrt{\log k} ) \le \frac{2\lambda k/\sqrt{\log k}}{(\lambda k/\sqrt{\log k})^2} 
\le \frac{2 \sqrt{\log k}}{\lambda k} \le 1/\sqrt{\log k}
$$
if $k$ is large.

For (ii) we use the supermartingale from Lemma \ref{super2}, which is the same as the one from Lemma \ref{super}, and simplify formulas as in the proof of Lemma \ref{life2}. If $q=P_{K,1}(T_{0,0} < T_L^+)$ then for $\lambda$ small optional stopping theorem gives
$$
q\le (1+\lambda/2)^{-\lambda k/\sqrt{\log k}} \le \exp(- \lambda^2 k/2\sqrt{\log k} ).  
$$

For (iii), we follow the argument in Lemma \ref{ignite}. We return to continuous time and note that by \eqref{Ydrift} the drift is 
$$
\le \mu = \lambda(1-p)k - p k - 1/\lambda
$$
so $Y_t - \mu t$ is a submartingale before time $V_L = T_{0,0} \wedge T^+_L$. Using the optional stopping theorem as before we conclude 
$$
E(V_L)\le \frac{L}{\mu} = \frac{(1-\ep)\lambda k}{1+\lambda} \cdot \frac{1}{ \lambda(1-p)k - p k - 1/\lambda}.
$$ 
Recalling the definition of $p$ 
\begin{align*}
\lambda(1-p)k - p k & = \lambda \left[ k - \frac{(1-\ep)\lambda k}{1+\lambda}\right] - \frac{(1-\ep)\lambda k}{1+\lambda} - \frac{1}{\lambda} \\
& = \lambda \left[  \frac{\ep\lambda k}{1+\lambda}\right] + \frac{\ep \lambda k}{1+\lambda} - \frac{1}{\lambda}.
\end{align*}
The first term is much smaller than the second so multiplying by $\lambda/\lambda$
$$
\frac{L}{\mu} \sim \frac{\lambda^2 k}{ \ep \lambda^2 k - (1+\lambda)} \sim 1/\ep,
$$ 
since $\lambda k^2 \to\infty$.
\end{proof}

\subsection{Power law graphs}

Suppose $P( d(x) \ge  m ) = 3^a m^{-a}$ for $m \ge 3$, where $a>2$ so that $Ed(x)^2<\infty$. In this case, the maximum degree vertex
on a graph with $n$ vertices  is $\sim n^{1/a}$, so the maximum eigenvalue $\Lambda \sim n^{1/2a}$ 
and the formula in \eqref{Wangcr} predicts that $\lambda_c \approx n^{-1/2a}$.
To prove an upper bound on $\lambda_c$ that is close to this, we suppose that $\lambda_0 = n^{-(1-2\eta)/2a}$. 

If $d(x) \ge k = n^{(1-\eta)/a}$ we call the vertex $x$ a star. 
$$
P(d(x) \ge  n^{(1-\eta)/a}) = 3^an^{-(1-\eta)}
$$ 
so if $n$ is large there are $\ge n^\eta$ stars with high probability. Now $\lambda_0^2 k = n^{\eta/a}$. By the estimate in the Lemma \ref{life2}, 
each individual star survives for time 
\beq
\ge \exp((1-4\eta)n^{\eta/a}/4).
\label{onestar}
\eeq
with probability $\ge 1 - 7 \exp(-(1-3\ep)n^{\eta/a}/4)$. The time 
$$
2r \kappa \le (4\nu \log n)\exp( O(\log ^2 n) )
$$ 
so \eqref{xfer} implies that with high probability the chosen star will transfer its infection to
its target by time $2r\kappa$ and we conclude that with high probability no lit star will die out during the process.

Combining these estimates shows that if $n$ is large then the number of infected stars $Y_k$ at time $2r\kappa k$ dominates 
a discrete time random walk that goes up by 1 with probability $p > e/(e+1)$ and down by 1 with probability $1-p$. Let $M\ge n^{\eta}$ be the number of stars.
Recalling that $((1-p)/p)^x$ is  harmonic function for a simple random walk that jumps up 1 with probability $p$, 
and down 1 with probability $1-p$ random walk, we see that $\exp(-Y_k)$ is a supermartingale while $Y_k \in (0,M)$, so 
$$
P_{0.9M}(T_0 < T_M) \le e^{-0.9M}.
$$
Since each cycle takes at least $0.1M(2r\kappa)$ units of time, we have survival for time $\exp(O(n^{\ep}))$
for some $\ep>0$.  

\subsection{Stretched exponential}

Suppose $P( d(x) \ge  m ) = \exp(-m^{1/b} + 3^{1/b})$ for $m \ge 3$, where $b>1$. In this case, the maximum degree vertex
on a graph with $n$ vertices  is $\sim \log^b n$, so the maximum eigenvalue $\Lambda \sim \log^{b/2} n$ 
and the formula in \eqref{Wangcr} predicts that $\lambda_c \approx \log^{-b/2} n$.

If $d(x) \ge k = \eta^b \log^b n$ we call the vertex $x$ a star. 
$$
P(d(x) \ge  \eta^b \log^b n) = \exp(3^{1/b})n^{-\eta},
$$ 
so if $n$ is large then the number of stars is $\ge n^{1-\eta}$ with high probability.

To see what value to take for $\lambda$ in our lower bound, we set the survival time equal to 1 over the probability of a
successful push, that is
$$
\exp(\lambda^2 \log^b n ) = (2/\lambda)^{2\nu \log n},
$$
or taking logs and rearranging
$$
\frac{\lambda^2}{\log(2/\lambda)} = 2\nu \log^{1-b} n.
$$
This means that the best upper bound we can hope to get is $\lambda_0 = (\log n)^{(1-\eta)(1-b)/2}$ versus the predicted value 
of $\log^{-b/2} n$.

By our choices we have
$$
\lambda_0^2 k =\eta^b( \log n)^{1+\eta(b-1) } 
$$
so Lemma \ref{life2} implies that the star survives for time 
$$
\ge \exp( (1-4\eta)\eta^b (\log n)^{1+\eta(b-1)} /4 )
$$
with probability $ \ge 1 -7 \exp( -(1-3\eta)\eta^b (\log n)^{1+\eta(b-1)} /4 )$. The time
$$
2r\kappa \le (4\nu \log n) \exp( \log n \cdot O(\log\log n))
$$ 
so \eqref{xfer} implies that with high probability the chosen star will transfer its infection to
its target by time $2r\kappa$ and we conclude that with high probability no lit star will die out during the process.
Comparing with random walk as in the previous proof, we have survival for time  $\exp(O(n^{1-\ep}))$ for any $\ep>0$.

\clearp

\end{document}